\documentclass[10pt]{article}

\usepackage[left=2cm,right=2cm,top=2cm,bottom=2cm]{geometry}

\usepackage{subfig}
\usepackage{epsfig} 
\usepackage{times} 
\usepackage{euscript}
\usepackage{tabularx}
\usepackage{mathtools}
\usepackage[font=footnotesize]{caption}
\usepackage{url}
\usepackage{enumitem}

\usepackage{moreverb}
\usepackage{amssymb,amsmath}
\usepackage{graphicx}
\usepackage[english]{babel}
\usepackage{color}
\usepackage{algorithmicx,algpseudocode}
\usepackage{algorithm}
\usepackage{algcompatible}
\makeatletter
\algnewcommand{\LineComment}[1]{\Statex \hskip\ALG@thistlm \(\triangleright\) #1}
\makeatother

\newtheorem{theorem}{Theorem}
\newtheorem{definition}[theorem]{Definition}

\newtheorem{ass}[theorem]{Assumption}

\newtheorem{lemma}[theorem]{Lemma}
\newtheorem{remark}[theorem]{Remark}

\newtheorem{proof}[theorem]{Proof}

\def\A{\mathcal{A}}
\def\E{\mathcal{E}}
\def\G{\mathcal{G}}
\def\N{\mathcal{N}}
\def\V{\mathcal{V}}
\def\1{{\mathbf{1}}}
\def\Exp{\mathbb{E}}
\def\Prob{\mathbb{P}}
\def\R{\mathbb{R}}
\def\C{\mathbb{C}}
\def\diag{\mathrm{diag}}
\def\blkdiag{\mathrm{blkdiag}}
\def\T{^{\top}}

\newcommand{\neighbors}[2]{ \mathcal{N}^{+}_{#1}(#2)}
\newcommand{\until}[1]{\left\{1,\ldots ,#1\right\}}

\newcommand{\tq}[1]{\textquotedblleft #1\textquotedblright}
\newcommand{\robustCost}[1]{\| #1 \|_{1,\nu}}

\newcommand{\FullTitle}{Generalized gradient optimization over lossy networks for partition-based estimation} 
\newcommand{\FullAuthors}{M.~Todescato, N.~Bof, G.~Cavraro, R.~Carli, L.~Schenato}

\title{\FullTitle}

\author{\FullAuthors}

\begin{document}

\footnotetext[1]{M.~Todescato, N.~Bof, R.~Carli and L.~Schenato are with the Department of Information Engineering, University of Padova, Italy, 35031. E-mail: [todescat,bofnicol,carlirug,schenato]@dei.unipd.it.}
\footnotetext[2]{G.~Cavraro is with the Virginia Polytechnic Institute ans State University, VA, USA. E-mail: cavraro@vt.edu.}

\maketitle

\begin{abstract}
We address the problem of distributed convex unconstrained optimization over networks characterized by asynchronous and possibly lossy communications. We analyze the case where the global cost function is the sum of locally coupled local strictly convex cost functions. As discussed in detail in a motivating example, this class of optimization objectives is, for example, typical in localization problems and in partition-based state estimation. Inspired by a generalized gradient descent strategy, namely the block Jacobi iteration, we propose a novel solution which is amenable for a distributed implementation and which, under a suitable condition on the step size, is provably locally resilient to communication failures. The theoretical analysis relies on the separation of time scales and Lyapunov theory. In addition, to show the flexibility of the proposed algorithm, we derive a resilient gradient descent iteration and a resilient generalized gradient for quadratic programming as two natural particularizations of our strategy. In this second case, global robustness is provided.
Finally, the proposed algorithm is numerically tested on the IEEE 123 nodes distribution feeder in the context of partition-based smart grid robust state estimation in the presence of measurements outliers.
\end{abstract}

\section{Introduction}\label{sec:intro}
The widespread of smart wireless electronic devices with the consequent creation of large-scale cyber-physical networked systems promises a new revolution in many fields. However, these novel engineering systems and the advent of the \tq{Big-Data} era require the development of new computational paradigms, due to the increasing amount of devices and data to be consistently managed. For example, many problems can be cast as optimization problems. As so, in the last years there has been a growing attention to distributed optimization tools which have become so important for two different reasons: first, the advent of Big Data asks for \emph{parallelisation} of the computational burden among many processing units since it is inconceivable to run optimization algorithms on one single (super)-computer. Second, many optimization problems are \emph{sparse} by nature since correlation between data is local. 
Nevertheless, one of the major hurdle to effectively deal with distributed optimization using multiple processing units is to guarantee synchronous and reliable communication. Indeed, communication can be wireless and CPU execution times might not be known in advance as in the context of cloud-computing. 
For this reason, although distributed optimization has a long history in the parallel and distributed computation literature, see, e.g., \cite{BD:TP:1989}, it has mainly focused on synchronous algorithms. However, to suitably fit with the upcoming large-scale system scenario, in the last years it has been reconsidered from a new peer-to-peer perspective.
The first class of algorithms appearing in this new literature relies on primal sub-gradient or descent iterations, as in \cite{NA:OA:2009,NA:OA:2010,Marelli2015}, which have the advantage to be easy to implement and suitable for asynchronous computation.
In order to induce robustness in the computation and improve convergence speed, augmented lagrangian algorithms such as the Alternating Direction Methods of Multipliers (ADMM) have been recently proposed. A first distributed ADMM algorithm was proposed in \cite{SI:RA:GG:2008,KV:GG:2013,BS:CR:TM:2014}, while a survey on this technique is \cite{BS:PN:CE:PB:EJ:2011}. However, for a distributed implementation, ADMM usually requires problems with very specific structures. In fact, most of the ADMM distributed algorithms are based on a consensus iteration \cite{EW:AO:2013}. Thus, a common drawback of this technique is that each node must store in its local memory a copy of the entire state vector. To avoid this problem, a recent partition-based and scalable approach applied to the ADMM algorithm is presented in \cite{ET:2012}, while to comply with asynchronous computation, suitable modification of the ADMM algorithm have been proposed in \cite{IF:BP:CP:HW:2013,BP:HW:IF:2014}. Finally, distributed algorithms based on Newton methods have been proposed to speed-up the computation \cite{ZM:RA:OA:JA:2014,ZF:VD:CA:PG:SL:2011}.\\

\noindent In this paper, we address the problem of \emph{distributed} convex unconstrained optimization over networks characterized by asynchronous and possibly lossy communications. We analyze the case where the global cost function is the sum of \emph{locally coupled} local costs. More specifically, by  \tq{locally coupled} and \tq{distributed} we mean the following
\begin{definition}[\bf Local coupling and distributed algorithm]\label{def:distributed}
Consider a set of $N$ processing units, $\until{N}$, which are interconnected according to a certain communication network. To each unit $i\in\until{N}$ a local cost $J_i$ is assigned. We say that $J_i$ is \emph{locally coupled} according to the communication network, if $J_i$ depends only on quantities which are related to unit $i$ and to units directly connected to it. In this case, a \emph{distributed algorithm} is a procedure running over the communication network and among the processing units, which only requires the exchange of local information among connected units. Differently said, with a slight abuse of nomenclature, a distributed algorithm is defined as a locally coupled procedure.\hfill$\square$ 
\end{definition}
\noindent Given Definition~\ref{def:distributed}, this study is motivated mainly by two facts. The first is its practical engineering relevance. Indeed, as discussed in detail in a motivating example we provide in Section~\ref{sec:example}, the structure of the class of convex optimization problems we analyze, characterizes a large variety of applications such as multi-area electric grid state estimation \cite{AC:SD:MC:2007,BS:CR:TM:2014}, localization in multi-robots formation \cite{Carron:14} and sensors networks \cite{NB:MT:RC:LS:2016} and Network Utility Maximization \cite{DPP:MC:2006}. The second is due to the class of gradient-based algorithms (e.g., \cite{NA:OA:2009,NA:OA:2010,Marelli2015}) we consider to solve our optimization problem. In particular, while it has the advantage to be easy to implement and suitable for asynchronous implementations, this class usually does not lead to \tq{distributed} solutions as intended in Definition~\ref{def:distributed}. Indeed, the derivatives of costs obtained as the sum of locally coupled costs are, usually, not locally coupled, yet they depend on information related to multi-hop processing units. Hence, the local functional dependance cannot be directly exploited. To overcome this issue, in some cases (\cite{NA:OA:2009,NA:OA:2010}) the algorithms require the local exchange of global information, hence all the processors eventually reach consensus to an optimal solution. In others, the algorithms require multiple communication rounds within the same algorithmic iteration (\cite{Marelli2015}). However, this solution implicitly asks for synchronicity. Hence, to deal with the case of lossy communications, a natural approach is to make the processing units store the last successfully received information from the neighboring units in order to leverage the vast existing body of literature on the so called \emph{partially asynchronous iterative methods} \cite{BD:TP:1989}. However, as later described in Section~\ref{sec:example}, in this scenario, because of packet drops and communication failures, the same state variables happen to appear in multiple delayed version. Thus, it is not possible to write the evolution of the state variables as a \emph{partially asynchronous iterative methods}. Finally, regarding the problem of computing non-locally coupled derivatives, another used approach is to exploit hyper-communication graphs which differs from the graph structure induced by the local coupling characterizing the cost function, thus artificially bypassing the limitations due to a \tq{truly} distributed procedure. As a particular example of this fact, consider, for instance, the case where the processors communicate through a communication network with star topology. According to Definition \ref{def:distributed}, each peripheral node can communicate only with the central node, while the central processor can communicate with everyone else. Conversely, if a two-hop communication is exploited, then the communication network turns out to be described by an \emph{all-to-all} topology.\\ 
In this regard, the main contribution of the paper is a truly distributed algorithm, based on a modified \emph{generalized gradient descent} iteration which, under suitable assumptions on the step size, is provably convergent and which is resilient to the presence of packet losses in the communication channel.
To the best of the authors' knowledge, this is one of the first provably convergent algorithms in the presence of packet losses, since even if both ADMM algorithms and distributed sub-gradient methods (DSM) can handle asynchronous computations, they still require reliable communication and usually do not satisfy Definition~\ref{def:distributed}. Interestingly, the proposed algorithm is also suitable for fully parallel computation, i.e., multiple agents can communicate and update their local variable simultaneously, and for broadcast communication, i.e., nodes do not need to enforce a bidirectional communication such as in gossip algorithms, and therefore is very attractive from a practical point of view.
It is anticipated that we presented the proposed algorithm in two preliminary versions. In \cite{MT:GC:RC:LS:2015} for the specific case of quadratic programming, while in \cite{NB:MT:RC:LS:2016} for the specific application of sensors networks locations. The algorithm is inspired on and resembles the block Jacobi iteration appeared in \cite{BD:TP:1989,SB:CL:1994}. However, in \cite{BD:TP:1989}, since the authors are majorly interested in parallel computation rather than implementing a distributed procedure suitable for today's sensors networks, they implicitly assume to exploit an hyper-communication graph. Conversely, in \cite{SB:CL:1994} the particular local dependency considered in the cost function ensures that first and successive derivatives are locally coupled.\\
To show the flexibility of the proposed procedure, we derive a resilient gradient descent iteration and a resilient generalized gradient for quadratic programming as two natural particularizations of our strategy. In this second case, we are able to provide global robustness.\\
Finally, we numerically study our algorithm on the standard IEEE 123-nodes test feeder for robust state estimation in the presence of measurements outliers.\\  
The rest of the paper is organized as follows: the rest of this section is devoted to the necessary notation and preliminaries. In Section \ref{sec:problem_formulation} we formulate the problem. In Section \ref{sec:example} we provide a motivating example for the proposed set-up. In Section \ref{sec:sync_communication} we analyze the case of synchronous and ideal communications. In Section \ref{sec:async_communication} we analyze the case of asynchronous and possibly unreliable communications. In Section \ref{sec:simulations} we test our algorithm. Finally, we present some concluding remarks in Section \ref{sec:conclusion}.

\subsection{Mathematical Preliminaries}\label{subsec:math_prem}
In this paper, $\mathcal{G}\left(\V,\E\right)$
denotes a directed graph where $\V = \until{N}$ is the set of vertices and $\E\subseteq\V\times\V$ is the set of directed edges. More precisely the edge $(i,j)$ is incident on node $i$ and node $j$ and is assumed to be directed away from $i$ and directed toward $j$. The graph $\G$ is said to be bidirected if $(i,j) \in \E$ implies $(j,i) \in \E$. 
Given a directed graph $\G\left(\V,\E\right)$, a directed path in $\G$ consists
of a sequence of vertices $\left(i_1, i_2,\ldots, i_r\right)$ such that $\left(i_j , i_j+1\right) \in \E$ for every $j \in \until{r-1}$. The length of a path is the number of directed edges which it consists of.
The directed graph $\G$ is said to be \emph{strongly connected} if for any pair of vertices $(i, j)$ there exists a directed path connecting $i$ to $j$. 
Given the directed graph $\G$, the set of neighbors of node $i$, denoted by $\N_i$, is given by $\N_i=\left\{j \in \V \,|\, (i,j) \in \E\right\}$. Moreover, $\N_i^+ = \N_i \cup \left\{i\right\}$.
Let us denote the cardinality of $\N_i^+$ by $\mu_i$, while the $j$-th neighbor of $i$ by $\neighbors{i}{j}$.
Given a directed graph $\G\left(\V,\E\right)$ with $|\E| =M$, let the \emph{incidence matrix} $\A \in \mathbb{R}^{M \times N}$ of $\G$ be defined as $\A=[a_{ei}]$, where $a_{ei}=1, -1, 0$, if edge $e$ is incident on node $i$ and directed away from it, is incident on node $i$ and directed toward it, or is not incident on node $i$, respectively.
Given a vector or a matrix, with $(\cdot)\T $ we denote its transpose, while with $\Re (\cdot)$ and $\Im (\cdot)$ its real and imaginary parts, respectively. Given a vector $v$, with $\diag(v)$ we denote the diagonal matrix whose diagonal elements are equal to the elements of $v$. Given a matrix $V$, with $\diag(V)$ we denote the vector obtained with the diagonal elements of $V$. Given a group of matrices $V_1,\ldots,V_n$, with $\blkdiag(V_1,\ldots,V_n)$ we denote the block diagonal matrix whose $i$-th block diagonal element is equal to $V_i$. Moreover, we denote with $\A_d:=\A\T \A$ the \emph{adjacency matrix} or \emph{laplacian matrix} of $\G$ which has the property  that $[\A_d]_{ij}\neq 0$ if and only if $(i,j) \in \E$.
If we associate to each edge a weight different from one, then it is possible to define the \emph{weighted laplacian matrix} as $\mathcal{L}=\A\T W\A$, where $W\in\mathbb{R}^{M\times M}$ represents the diagonal matrix containing in its $i$-th element the weight associated to the $i$-th edge. We will also consider \emph{strictly convex functions} $f(x):\R^n\to\R$, i.e., $\forall \,x_1\neq x_2$ and $\eta\in(0,1)$ then $f(\eta x_1+(1-\eta)x_2)<\eta f(x_1)+(1-\eta)f(x_2)$ and \emph{radially unbounded}, i.e. $\|x\|\to +\infty \Rightarrow f(x)\to \infty$. 
Finally, with the symbols $\Exp$ and $\Prob$ we denote, respectively, the expectation operator and the probability of an event.

\section{Problem Formulation}\label{sec:problem_formulation}
Consider a set of $N$ agents $\V =\{ 1,\ldots,N\}$, where each agent $i\in\V$ is described by its state vector $x_i\in\R^{n_i}$. Assume the agents can communicate among themselves through a bidirected strongly connected \emph{communication graph} $\G(\V,\E)$. In this paper, we are interested in extending to the more general case of convex costs the algorithm first presented in \cite{MT:GC:RC:LS:2015} for the case of quadratic programming and in \cite{NB:MT:RC:LS:2016} for the specific application of sensors networks localization. In particular, we examine a particular class of separable strictly convex cost functions which exhibit local and possibly nonlinear dependence among the states of neighboring nodes. By defining the overall state vector as $x=[x_1\T ,\ldots,x_N\T ]\T \in\R^n$ ($n=\sum_i n_i$), we consider the following optimization problem 
\begin{equation}\label{eq:ConvexProblem}
\min_x J(x) \equiv \underset{x_1,\ldots,x_N}{\min}\ \sum_{i=1}^N J_i(x_i,\{x_j\}_{j\in\N_i})\, .
\end{equation} 
Observe that the local dependence coincides with the communication graph $\G$, i.e., each cost function $J_i$ depends on information regarding only agent $j\in\N_i^+$.\\
We will consider the following assumption on the cost fucntions:
\begin{ass}[{\bf Strict convexity and radial unboundedness}]\label{ass:convex}
The function $J(x)$ is assumed to be strictly convex and radially unbounded.\hfill$\square$ 
\end{ass}
Observe that under the previous assumption the minimizer $x^*$ of Problem~\eqref{eq:ConvexProblem}  exists and is unique
\begin{equation}\label{eq:minimizer}
x^* := \underset{x}{\mathrm{argmin}}\ J(x)\, ,
\end{equation} 
but the local costs function $J_i$ do not need to be strictly convex and radially unbounded. Indeed in many estimation problems the local cost functions $J_i$ are just strictly convex but not radially unbounded. 
The standard approach to solve the previous optimization problem is to resort to some centralized iterative algorithm acting on $J$, e.g., Newton-Raphson, which makes use of global knowledge of the network' states, costs and topology. On the contrary, by leveraging the particular local dependence characterizing each cost function $J_i$, we want to solve Problem~\eqref{eq:ConvexProblem} by developing a procedure which is \emph{distributed}, i.e., exploiting only local exchange of information among neighbors, and \emph{resilient}, i.e., resilient to communication limitations and non idealities.\\
We will also use the following simplified notation for local components of gradients and hessians:
$$ \nabla_i J_j = \frac{\partial J_j}{\partial x_i}, \ \ \ \ \   \nabla^2_{i\ell} J_j = \frac{\partial^2 J_j}{\partial x_i \partial x_\ell}\, .$$

\begin{remark}[\textbf{On the class of separable cost functions}]\label{rem:separable_cost}
The class of functions considered can arise in diverse applications such as state estimation in smart electric grids \cite{MT:GC:RC:LS:2015} and sensor networks localization \cite{NB:MT:RC:LS:2016}, just to mention some of them. In the particular case of quadratic cost, the optimization problem falls onto the standard linear least-squares framework. Nevertheless, as it will be shown in the simulation Section~\ref{sec:simulations}, the class is much more general and comprises penalty functions used, e.g., to perform robust statistics and general nonlinear least-squares optimization. Of particular interest is the more general framework of parallel computation in optimization. For both privacy and efficiency reasons, the computational burden can be split among several distributed machines. To each of them only information about $J_i$ is assigned. Thanks to local exchange of information, the machines must distributely compute a solution of \eqref{eq:ConvexProblem}.\hfill$\square$   
\end{remark}
\begin{remark}[\textbf{Partition-based modular communication architecture}]\label{rem:modularity}\ \\
Note that the particular communication architecture considered, seamlessly describes the case of communications among single peer agents as well as among  large areas consisting of a collection of peers. The only difference relies on the particular definition of the set $\V$ and of the agents' state $x_i$. For instance, in the case of sensor localization, each sensor might represent an agent of $\V$ while $x_i$ might describe its absolute position in an inertial global reference frame. Conversely, in the case of smart grids state estimation, one agent might describe an entire electric feeder; then, $x_i$ would be either voltages or currents at all the electric buses of the corresponding feeder.
\hfill$\square$
\end{remark}
%

\section{Motivating example: State estimation in Smart Power Distribution Grids}\label{sec:example}

In steady state the voltages and currents in a power distribution grid are regulated by the Kirchhoff's laws  which can be written as follow:
$$ 
Lv = i^c\, , 
$$
where $L$ is the admittance matrix, and $v$ and $i^c$ are the vector collecting all the $N$ voltages and currents of the nodes in the grid, respectively. The admittance matrix is a sparse matrix, in the sense that the current at a specific node $i$, namely $i^c_i$, depends only on its own voltage and the voltages of its physically connected neighbour nodes $\mathcal{N}_i$, i.e. 
$$ 
i^c_i = \sum_{j\in \mathcal{N}^+_i}L_{ij}v_j\, .
$$
In future smart distribution grids, it is expected that each node $i$ would be able to take noisy measurements of its voltage and current, i.e. 
\begin{eqnarray*}
y_i^v &=& v_i + w^v_i\, , \\
y_i^{i^c} &=& i^c_i + w^{i^c}_i = \sum_{j\in \mathcal{N}^+_i}L_{ij}v_j + w^{i^c}_i\, ,
\end{eqnarray*}
where $ w^v_i, w^{i^c}_i$ represent the measurement noise for the voltage and current measurements, respectively. It is also expected that these nodes are embedded with communication capabilities, such as power line communication (PLC), which allow them to communicate with their physically connected neighbours. As so the communication network and the physical network will coincide. 
The (centralized) state estimation problem is the process that, given all the measurements $\{ y_i^v, y_i^{i^c}\}_{i=1}^N$, should return the best estimate of all the voltages and currents  $\{ v_i, i^c_i\}_{i=1}^N$. The standard approach is to cast this problem as a least-square estimation problem, where the unknown quantities to be estimated are the voltages $v^*$, since the currents can be estimated directly from the voltages via the Kirchhoff's law $i^{c*}=Lv^*$. 
In this work, we are interested in solving this problem in a distributed fashion via a partition-based communication architecture. For the sake of clarity, let us assume that the grid is divided into $N$ partitions each corresponding to a node. To each partition, we associate the corresponding voltage, which we collect in the vector $x_i\in \mathbb{R}$\footnote{In reality, the voltages and currents in steady state are phasors, i.e., should be represented as complex numbers. However, the discussion in this section can be extended w.l.o.g. also to the more realistic scenario, which is indeed considered in the Simulation section below.}. 
Let us also define with $y_i= [y_i^v \ y_i^{i^c}]\T  \in \mathbb{R}^2$ and $w_i= [w_i^v \ w_i^{i^c}]\T  \in \mathbb{R}^2$ the measurement vector and the measurement noise corresponding to the measurements of the voltage and current at node $i$. Let us also define the vectors $x=[x_1 ,\dots, x_N]\T  \in \mathbb{R}^N$, $y=[y_1\T, \dots, y_N\T]\T \in\mathbb{R}^{2N}$, $w=[w_1\T, \dots, w_N\T]\T  \in \mathbb{R}^{2N}$. As so the measurement model can be written as:
$$ 
y_i = \sum_{j=1}^NA_{ij} x_j +w_i = \sum_{j\in\mathcal{N}_i^+}A_{ij} x_j +w_i \ \ \  \ \ \mbox{($A_{ij}=0$ if $j\notin \mathcal{N}_i^+$)}\, ,
$$
where $A_{ij}$ can be easily be obtained from the elements of the matrix $L$, or equivalently in vector form
$$ 
y = A x+w\, ,
$$
where $A= [A_1\T,  \dots, A_N\T ]\T \in\mathbb{R}^{2N\times N}$ and $A_i=[A_{i1}, \dots, A_{iN}]\in\mathbb{R}^{2\times N}$. If we define
$$ 
J_i(x_i,\{x_j\}_{j\in\mathcal{N}_i})=\frac{1}{2}\|y_i-A_i x\|^2, \ \ J(x) = \sum_{i=1}^N  J_i(x_i,\{x_j\}_{j\in\mathcal{N}_i})=\frac{1}{2}\|y-Ax\|^2\, ,
$$
with
\begin{align*}
&\nabla J(x) = A\T (Ax - y)\,, \qquad 
\nabla^2 J(x) = H= A\T A\,,\\
&H_{ij}=\sum_{\ell=1}^N A\T _{\ell i}A_{\ell j}= \sum_{\ell\in\mathcal{N}_i^+} A\T _{\ell i}A_{\ell j} = \sum_{\ell\in(\mathcal{N}_i^+ \cap \mathcal{N}_j^+)} A\T _{\ell i}A_{\ell j} 
\end{align*} 
the optimal (centralized) least squares solution\footnote{The formulation can be extended to the weighed least square solutions if noise with different variances $R$ are included which would lead to the solution $x^*=(A\T R^{-1}A)^{-1}A\T R^{-1}y$, but for the sake of clarity in the notation of this section, it is omitted.} is given by:
$$
x^* =\mathrm{argmin}_{x} J(x) = (A\T A)^{-1}A\T y\, . 
$$
A standard approach to asymptotically obtain the optimal solution is to employ an iterative algorithm based on the generalized gradient descent:
$$ 
x^+ = x -\epsilon D^{-1}A\T (Ax-y)= x -\epsilon D^{-1} \nabla J(x)=  x -\epsilon D^{-1}(Hx - A\T y)\, ,
$$
where $\epsilon$ is a suitable stepsize and $D$ is a strictly positive definite matrix, i.e. $D>0$. A typical way to solve the previous update in a distributed fashion is to pick a block-diagonal matrix $D$, i.e. $D=\blkdiag(D_1,\ldots,D_N)$, so that the previous centralized update can be written as
\begin{eqnarray} 
x_i^+ &= &x_i - \epsilon D_i^{-1}(\sum_{j=1}^N H_{ij}x_j - \sum_{j=1}^N A_{ji}\T y_j)\notag\\ 
&=& x_i - \epsilon D_i^{-1}(\sum_{j=1}^N \sum_{\ell=1}^N A\T _{\ell i}A_{\ell j} x_j - \sum_{j=1}^N A_{ji}\T y_j) \notag\\
 &=& x_i - \epsilon D_i^{-1}( \sum_{j\in \mathcal{N}_\ell^+, \forall \ell\in  \mathcal{N}^+_i} A\T _{\ell i}A_{\ell j} x_j - \sum_{j\in\mathcal{N}^+_i} A_{ji}\T y_j)\,, \label{eqn:GGD_synch}
\end{eqnarray}
where we exploited the property that $A_{ij}=0$ if $j\notin \mathcal{N}_i^+$. While the second summation involves only measurements that belongs to the neighbours of node $i$, the first summation requires the node $i$ to collect the state variables $x_j$ that belongs to the neighbours of the neighbours. As so, this implementation is not really distributed, since two-hop communication is required. Although this is not impossible from a practical perspective, it requires substantial additional communication and synchronization efforts. An alternative approach which allows the implementation of a truly distributed algorithm is to create the additional local variable at each node~$i$:
$$
z_i = A_i x_i = \sum_{j\in\mathcal{N}_i^+}A_{ij} x_j\, ,\quad  \forall i\,,
$$
which can be collected in the vector $z=[z_1\T,  \dots, z\T _N]\T $, so that in matrix form the previous expression can be written as $z=Ax$. With this notation the generalized gradient descent can be written as:
\begin{eqnarray*}
z_i^+ &=& \sum_{j\in\mathcal{N}_i^+}A_{ij} x_j \\
x^+_i &=& x_i - \epsilon D_i^{-1}(\sum_{\ell=1}^N A\T _{\ell i} \underbrace{\sum_{j=1}^N A_{\ell j} x_j}_{z_\ell} - \sum_{j=1}^N A_{ji}\T y_j)\notag\\ 
&=& x_i - \epsilon D_i^{-1}\sum_{j\in\mathcal{N}^+_i} A_{ji}\T (z^+_i - y_i)\, .
\end{eqnarray*} 
This alternative solution requires two communication rounds to compute $x^+_i$, since first it is necessary to send the $x_i$ to compute $z_i^+$, and then to transmit $z_i$ to the neighbours\footnote{It is necessary to transmit the measurements $y_i$ only at the initialization phase since they do not change during the course of the evolution of the algorithm.}. In practical scenarios, such as using PLC protocols, synchronization of transmissions and updates can be difficult. Moreover packet losses might occur, i.e. some messages from the neighbours might not be received. A naive solution to this problem, is to use local registers that keep in memory the latest message received from the neighbours, and then use these values whenever an update of the local variables $x_i,z_i$ is needed. It can be shown that this is equivalent to a scenario where every node $j\in \mathcal{N}_i$ use a delayed version of the local variables $x_i,z_i$. Since the variables $z_i$ are function of the (possibly delayed) state variables $x_i$, the variables $x_i$ are themselves functions of the delayed version of variables $x_i$ of the network. More specifically, it can be shown that the previous update equations can be written as
\begin{eqnarray}
z_i(t+1) &=& \sum_{j\in\mathcal{N}_i^+}A_{ij} x_j(\tau'_{ij}(t))\, , \\
x_i(t+1) &=& x_i(t) - \epsilon D_i^{-1}\big(\sum_{\ell=1}^N A\T _{\ell i} \sum_{j=1}^N A_{\ell j} x_j(\tau_{\ell j}(t)) - \sum_{j=1}^N A_{ji}\T y_j \big)\, . \label{eqn:GGD_asych_new}
\end{eqnarray}
where $0\leq \tau_{ij}(t),\tau'_{ij}(t)\leq k$ represent the delay of each variable which depends on the specific sequence of packet losses and variable updates, and explicitly included the time dependency of each variable. Note that in the last equation the variable $x_j$ might appear with multiple instances with different delays into the update of the variable $x_i$, i.e. it is not possible to write the evolution of variables of the original generalized gradient descent algorithm given in Eqn.~\eqref{eqn:GGD_synch} as a \emph{partially asynchronous iterative methods}  (see chapter 7 of \cite{BD:TP:1989}), for which en extensive body of literature exists, since the cited framework would require the algorithm to be written as:
\begin{eqnarray} x_i(t+1) &= & x_i(t) - \epsilon D_i^{-1}\big(\sum_{j=1}^N H_{ij} x_j(\tau_{ij}(t)) - \sum_{j=1}^N A_{ji}\T y_j \big)\, .\label{eqn:GGD_asynch}
\end{eqnarray}
Motivated by this observation, in this work we will propose an alternative mathematical machinery based on Lyapunov theory and the separation of time scale principle to prove convergence of the asynchronous algorithm~\eqref{eqn:GGD_asych_new} for a sufficiently small stepsize $\epsilon$. Note that this machinery can also be applied to more general convex problems. This is useful, for example, in the presence of outliers or sensor faults in order to develop more robust estimators than least squares. In fact, a common way to enforce robustness in the estimation is to replace  the quadratic cost function defined above with the 1-norm of the residuals, that is 
\begin{equation}\label{eq:1norm}
J_i(x_i,\{x_j\}_{j\in\mathcal{N}_i})=\|y_i-A_i x\|_1\, .
\end{equation}
However, since \eqref{eq:1norm} is not differentiable, it cannot be directly used with our algorithm. To deal with this issue, in the Simulation section \ref{sec:simulations}, we will exploit the following modification of the 1-norm \cite{Argaez2011ell1norm}
\begin{equation}\label{eq:modified1norm}
\robustCost{\cdot}\ : \R^n\to\R\ ,\ \ x\mapsto \robustCost{x} :=\sum_{i=1}^N \sqrt{x_i^2 + \nu}\, ,
\end{equation}
where $\nu>0$ is such that the smaller the selected value of $\nu$ is, the better the approximation of the 1-norm is. In particular, the approximation of each term in the summation of the cost function is quadratic when $x_i$ belongs to a small neighborhood of $0$.\\
%
%
The next Sections will then provide a fully distributed generalized gradient descent algorithm which is resilient to lossy communication.

\section{Synchronous update and reliable communication}\label{sec:sync_communication}
In this section we analyze the case of synchronous and ideal, i.e., reliable, communications among neighbors, leaving the extension to the more realistic case of unreliable communication to Section~\ref{sec:async_communication}.\\

\noindent Consider the optimization Problem~\eqref{eq:ConvexProblem}. In the ideal communication case, one possible choice to iteratively solve Problem \eqref{eq:ConvexProblem} is to exploit the so called \emph{generalized gradient descent} iteration
\begin{equation}
x^+ = x - \epsilon D^{-1}(x) \nabla J(x) \, ,\qquad\qquad x(0) = x_0\, ,
\label{eq:gen_grad_descent}
\end{equation}
where $\nabla J(x):= \left[\frac{\partial J(x)}{\partial x}|_{x}\right]\T $ is the gradient of $J$ evaluated at the current value $x$, $D(x)$ is a generic positive definite matrix, possibly function of $x$ itself,  and $\epsilon$ a suitable positive constant, referred to as \emph{step size}. Observe that depending on the particular choice of $D(x)$, Eq.~\eqref{eq:gen_grad_descent} describes various types of algorithms. Indeed, if $D(x)=I$, the standard gradient descent iteration is obtained; if $D(x)$ is chosen to be diagonal with diagonal elements equal to those of the Hessian matrix, then a Jacobi descent iteration is retrieved; while, if $D(x)$ is equal to the entire Hessian, then Eq.~\eqref{eq:gen_grad_descent} returns the classical Newton-Raphson iteration.\\ 
The algorithm we propose (and describe in Section \ref{sec:async_communication}) is inspired by the particular case of \eqref{eq:gen_grad_descent}, referred to as \emph{block Jacobi}, where we choose $D(t)$ to be the \emph{block diagonal} matrix such that
\begin{equation}\label{eq:D-1}
D(x) = \mathrm{blkdiag}(D_1(x),\ldots,D_N(x))\, ,\qquad\qquad D_i(x) := \nabla^2_{ii} J(x)\, ,\quad i\in\V\, ,
\end{equation}
i.e., where each diagonal block coincides with the second order derivative of $J$ w.r.t. $x_i$. Thanks to this choice for the matrix $D$, Eq.~\eqref{eq:gen_grad_descent} can be split into partial state updates each of which equal to 
\begin{equation}\label{eq:areai_syncro_upd}
x_i^+ = x_i - \epsilon D_i^{-1}(x)\nabla_i J(x)\, ,\quad i\in\V\, .
\end{equation}
Now, it is convenient to explicitly take into account the separable structure of the cost function $J$ in order to show that each gradient block $\nabla_i J$ as well as each $D_i$ block can be computed exploiting only local information coming from agent's $i$ \emph{two-steps neighbors}, i.e., agents connected to agent $i$ by a directed path of length two. Indeed, for the gradient we have that
\begin{align}\label{eq:GradientTerm}
\nabla_i J(x) =  \sum_{j\in\N_i^+} \nabla_i J_j(\{x_k\}_{k\in\N_j^+}) =  \nabla_i J_i(x_i, \{x_j\}_{j\in\N_i}) +  \sum_{j\in\N_i}  \nabla_i J_j(x_j, \{x_k\}_{k\in\N_j}) ,
\end{align}
and it can be seen that the first term on the RHS of Eq.~\eqref{eq:GradientTerm} depends only on information coming from $j\in\N_i^+$; while, the second term depends on information coming from neighbors of node $i$ and from the neighbors of its neighbors, $k\in\N_j^+$. A similar reasoning applies to $D_i$ indeed,
\begin{align}
D_i(x)&:=\sum_{j\in\N_i^+} \nabla^2_{ii} J_j(\{x_k\}_{k\in\N_j^+})\notag\\
&=  \nabla^2_{ii} J_{i}(x_i, \{x_j\}_{j\in\N_i}) +  \sum_{j\in\N_i}  \nabla^2_{ii} J_j(x_j, \{x_k\}_{k\in\N_j})\,.\label{eq:HessTerm}
\end{align}
Again, the first term in the RHS of Eq.~\eqref{eq:HessTerm} depends only on node $i$ direct neighbors, $j\in\N_i^+$, while the second term requires information coming from the neighbors of its neighbors. 
In view of a distributed computation, we assume each agent $i\in\V$, once gathered the neighbors states $\{x_j\}_{j\in\N_i}$, can compute and store in its local memory, in addition to the state $x_i$, the following variables 
\begin{equation}\label{eq:local_memory}
\rho^{(j)}_i(x) := \nabla_j J_i(x_i,\{x_j\}_{j\in\N_i})\, ,\qquad\qquad
\xi^{(j)}_i(x) := \nabla^2_{jj} J_i(x_i,\{x_j\}_{j\in\N_i})\, ,
\end{equation}
which represent the partial components of the first and second derivatives of its local cost $J_i$ evaluated at the current state value. Observe that, since in a distributed framework each agent is assumed to have information only regarding its local cost $J_i$, the $\rho$'s and $\xi$'s variables represents the quantities which agent $i$ must compute and send to its neighbors in order to let them compute their corresponding gradient and hessian blocks. Likewise, agent $i$ needs to receive similar variables from each one of its neighbors. Indeed, thanks to Eqs.~\eqref{eq:GradientTerm}--\eqref{eq:HessTerm}, it holds that 
\begin{equation}\label{eq:LocalUpd}
\nabla_i J(x)= \sum_{j\in\N_i^+} \rho^{(i)}_j(x)\, ,\qquad\qquad
D_i(x) = \sum_{j\in\N_i^+} \xi^{(i)}_j(x)\, .
\end{equation}
As above stressed, each agent $i\in\V$, to iteratively compute \eqref{eq:areai_syncro_upd}, can perform its computations autonomously assuming it has at its disposal information coming from its two-steps neighbors. However, this presents two major drawbacks:
\begin{enumerate}
\item it clashes with a truly distributed setting which exploits the exchange of information only among one-step neighbors;
\item within successive iterations, to ensure consistency and thus convergence of the procedure to a minimizer of Problem~\eqref{eq:ConvexProblem}, all the communications must be synchronous and reliable.  
\end{enumerate}
To workaround the first issue one possible solution would be, at each iteration, to perform two communication rounds among one-step neighbors as illustratively shown in Figure~\ref{fig:sync_comm_scheme}. The first round is used to exchange the state values among neighboring agents in order them to compute all the partial information terms according to Eqs.~\eqref{eq:local_memory}--\eqref{eq:LocalUpd}; while the second round is used to communicate the computed variables in order to perform the state update as in Eq.~\eqref{eq:areai_syncro_upd}. Regarding the second issue, it necessarily enforces the use of suitable synchronization algorithms as well as re-transmission protocols in case of packet failures. What above described has been compactly written in algorithmic form as reported in Algorithm~\ref{alg:dbja} in which $\mathtt{flag_{\rm transmission}}$ denotes a variable to control communication and update among the agents. Note that, even if these might provide possible answers, it is understood they do not represent satisfactory solutions for real-world applications. Conversely, in the next section we propose a truly distributed and resilient iterative procedure which, by naturally exploiting information coming from one-step neighbors and being resilient to packet losses and communication non idealities, is much more appealing from an engineering perspective.

\begin{algorithm}
    \begin{algorithmic}[1]
   	 \REQUIRE  $x_i^o$, $\epsilon$
		\STATE $x_i \leftarrow x_i^o$
		\IF {$\mathtt{flag_{transmission}}=1$} 
		\STATE \textbf{Broadcast}: $x_i$
		\STATE \textbf{Receive}: $x_j, \  \ \forall j\in \mathcal{N}_i$
		\STATE $\rho^{(j)}_i \leftarrow  \nabla_j J_i(\{x_k\}_{k\in\N_j^+}), \ \ \forall j\in \mathcal{N}^+_i$
		\STATE $\xi^{(j)}_i \leftarrow  \nabla^2_{jj} J_i(\{x_k\}_{k\in\N_j^+}), \ \ \forall j\in \mathcal{N}^+_i$
		\STATE \textbf{Broadcast}: $\rho^{(j)}_i,\xi^{(j)}_i, \ \ \forall j\in \mathcal{N}_i$
		\STATE \textbf{Receive}: $\{\rho^{(i)}_j,\xi^{(i)}_j\}, \ \ \forall j\in \mathcal{N}_j$
		\STATE $x_i \leftarrow x_i - \epsilon \big(\sum_{j\in\mathcal{N}_i^+}  \xi^{(i)}_j \big)^{-1}\big(\sum_{j\in\mathcal{N}_i^+}  \rho^{(i)}_j \big)$
		\ENDIF
\end{algorithmic}
\caption{Distributed Block Jacobi algorithm (node $i$).}
\label{alg:dbja}
\end{algorithm}
\begin{figure}[t]
\centering
\includegraphics[width=\textwidth]{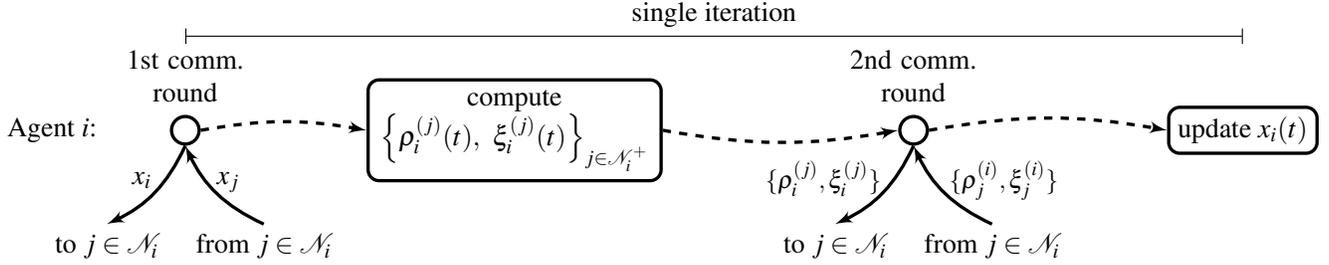}
\caption{Communication scheme to perform one single block Jacobi iteration \eqref{eq:areai_syncro_upd} in a distributed setting which assumes only information exchange among one-step neighbors.}
\label{fig:sync_comm_scheme}
\end{figure}

\section{Asynchronous updates and unreliable communication: the Resilient Block Jacobi (RBJ) algorithm}\label{sec:async_communication}
In this section we consider the more realistic case of asynchronous and unreliable communications where each agent might either receive asynchronous information coming from its neighbors, or not receive it. In particular, we present a modified iteration and analyze its corresponding iterative algorithm, which we refer to as \emph{resilient block Jacobi}, which (i) exploits only information coming from one-step neighbors; (ii) requires only one communication round per algorithmic iteration; (iii) is based on an asynchronous communication protocol; (iv) is resilient to communication failures. First, we present our algorithm for the general case of separable convex costs. Later, we particularize the algorithm to suit two special cases and showing its flexibility.\\

\noindent Consider the standard block Jacobi iteration \eqref{eq:areai_syncro_upd}. As analyzed in Section \ref{sec:sync_communication}, the procedure exhibits some fundamental criticisms which deeply compromise its distributed and asynchronous implementation and yet its robustness properties. Thus, to develop our algorithm, we need to suitably modify iteration \eqref{eq:areai_syncro_upd}. The modification we propose is apparently naive since the idea is to simply equip each agent with an additional amount of memory storage to keep track of the last received and available information corresponding to each neighbor. This additional memory is then used to perform Eq.~\eqref{eq:areai_syncro_upd}. Indeed note that, if agent $i$ does not receive some of the information coming from its neighbors, it does not have the necessary information to synchronously  compute neither \eqref{eq:local_memory} nor \eqref{eq:LocalUpd} and thus it is not able to update its state according to \eqref{eq:areai_syncro_upd}.\\
To model randomly occurring packet losses is convenient to introduce the indicator function
\begin{equation*}
\gamma^{(i)}_j(t) = \left\{\begin{array}{cl}
1 &\text{ if }i\text{ received the information sent by }j\text{ at iteration } t\\
0 &\text{ otherwise.}
\end{array}\right.
\label{eq:comm_fun}
\end{equation*}
with the assumption that $\gamma^{(i)}_i(t)=1$, since node $i$ has always access to its local variables. Then, as suggested above, the main idea is to equip each agent $i$ with auxiliary variables $\left\{\widehat{x}^{(i)}_j, \widehat{\rho}^{(i)}_j,\widehat{\xi}^{(i)}_j\right\}_{j\in\N_i}$, used to keep track of the last available information received from each neighbors. Specifically, the dynamic for the $j$-th set of additional variables is given by
\begin{equation}\label{eq:memory_dynamics}
\left\{\widehat{x}^{(i)}_j(t), \widehat{\rho}^{(i)}_j(t), \widehat{\xi}^{(i)}_j(t)\right\} = 
\left\{\begin{array}{ll}
\left\{x_j(t),\rho^{(i)}_j(t),\xi^{(i)}_j(t) \right\}\, , & \text{ if }\gamma^{(i)}_j(t) = 1\, ; \\
\\
\left\{\widehat{x}^{(i)}_j(t-1), \widehat{\rho}^{(i)}_j(t-1), \widehat{\xi}^{(i)}_j(t-1) \right\}\, , & \text{ if }\gamma^{(i)}_j(t) = 0\, .
\end{array}\right.
\end{equation}
Thanks to this additional memory at every algorithmic iteration, each agent can perform its local update which, inspired on Eq.~\eqref{eq:areai_syncro_upd}, becomes equal to
\begin{equation}\label{eq:areai_async_upd}
x_i(t+1) = x_i(t) - \epsilon\left(\sum_{j\in\N_i^+} \widehat{\xi}^{(i)}_j(t) \right)^{-1}\left(\sum_{j\in\N_i^+}\widehat{\rho}^{(i)}_j(t)\right)\, .
\end{equation}
Observe that the differences between Eqs.~\eqref{eq:areai_syncro_upd} and \eqref{eq:areai_async_upd} are mainly two:
\begin{enumerate}
\item the variables in agent $i$'s memory, used to store the first and second partial derivatives of $J_i$ w.r.t. $j\in\N_i$, are necessarily computed as
\begin{equation}\label{eq:GradHessComp_rBJ}
\rho^{(j)}_i(t) =  \nabla_j J_i(x_i(t),\{\widehat x^{(i)}_k(t)\}_{k\in\N_i} ),\quad 
\xi^{(j)}_i(t) = \nabla^2_{jj} J_i(x_i(t),\{\widehat x^{(i)}_k(t)\}_{k\in\N_i}),
\end{equation}
that is, they are evaluated at the last stored states' values; likewise, the values of the additional variables $\{\widehat{\rho}^{(i)}_j,\ \widehat{\xi}^{(i)}_j\}_{j\in\N_i}$ correspond to those last received from each neighbor and computed by each of them using the last available information on their neighbors' states; 
\item conversely to the synchronous implementation of the algorithm, at each iteration only one communication round is performed. This means the agents send only one packet per iteration, consisting of the state and the partial derivatives. See Figure~\ref{fig:async_comm_scheme} for an illustrative representation.
\end{enumerate}
\begin{figure}[t]
\centering
\includegraphics[width =  0.8\columnwidth]{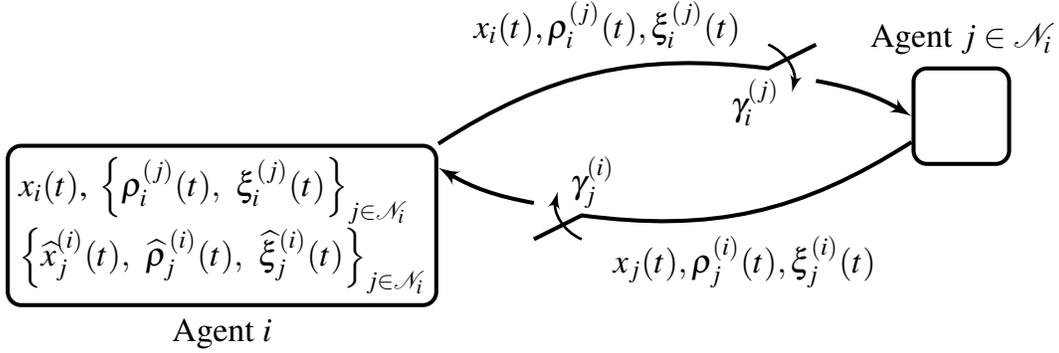}
\caption{\small Memory storage and communication scheme between pairs of neighbors agents for the RBJ algorithm.}
\label{fig:async_comm_scheme}
\end{figure}
Thanks to this simple modification the agents can perform their updates asynchronously and independently. Moreover, since only one communication round per iteration is required, both the communication burden and the number of possible communication failures are reduced. Nevertheless, it is worth stressing that, even if no packet losses occur, the classical block Jacobi and our resilient block Jacobi iteration does not exactly coincide. Indeed, in the resilient case, by sending only one packet per iteration, the state and the partial derivative information would be \tq{delayed} one from each other of one iteration if compared with the synchronous implementation.
The \emph{resilient block Jacobi} algorithm (hereafter referred to as RBJ  algorithm) for separable convex functions is formally described in Algorithm~\ref{alg:rbj} where it is presented in an event-based update performed by a generic node $i$. The variables  $\mathtt{flag_{transmission}},\mathtt{flag_{reception}},\mathtt{flag_{update}}$ are flag variables which determines which specific action a node is performing, namely transmission, reception or update. When each action is performed it cannot be interrupted, however the specific order or consecutive calls of an action do not impair the convergence of the proposed algorithm and therefore can be used independently of the specific communication protocol or CPU multitasking scheduling. 

\begin{algorithm}
    \begin{algorithmic}[1]
   	 \REQUIRE  $x_i^o$, $\epsilon$
	 	\Statex{ \bf Initialization} (atomic)
		\STATE $x_i \leftarrow x_i^o$
		\STATE $\widehat  x^{(i)}_j \leftarrow 0, \ \ \forall j\in \mathcal{N}_i $
		\STATE $ \rho^{(j)}_i \leftarrow 0, \ \ \forall j\in \mathcal{N}_i $
		\STATE $\widehat  \rho^{(i)}_j \leftarrow 0, \ \ \forall j\in \mathcal{N}_i $
		\STATE $ \xi^{(j)}_i \leftarrow I, \ \ \forall j\in \mathcal{N}_i $
		\STATE $ \widehat \xi^{(i)}_j \leftarrow I, \ \ \forall j\in \mathcal{N}_i $
		\STATE $\mathtt{flag_{transmission}}\leftarrow 1$ (optional)
		\Statex{ \bf Transmission} (atomic)
		\IF {$\mathtt{flag_{transmission}}=1$} 
		\STATE $\mathtt{transmitter\_node\_ID} \leftarrow i$
		\STATE $ \rho^{(j)}_i \leftarrow \nabla_j J_i(x_i,\{\widehat x^{(i)}_k\}_{k\in\N_i}), \ \ \forall j\in \mathcal{N}_i $
		\STATE $ \xi^{(j)}_i \leftarrow \nabla^2_{jj} J_i(x_i,\{\widehat x^{(i)}_k\}_{k\in\N_i}), \ \ \forall j\in \mathcal{N}_i $
		\STATE \textbf{Broadcast}: $\mathtt{transmitter\_node\_ID}, x_i, \{\rho^{(j)}_i,\xi^{(j)}_i\}_{j\in \mathcal{N}_i}$
		\STATE $\mathtt{flag_{transmission}}\leftarrow 0$
		\STATE $\mathtt{flag_{reception}}\leftarrow 1$ (optional)		
		\ENDIF
		\Statex{ \bf Reception} (atomic)
		\IF {$\mathtt{flag_{reception}}=1$} 
		\STATE $j\leftarrow \mathtt{transmitter\_node\_ID}$
		\STATE $ \widehat x^{(i)}_j \leftarrow  x_j $
		\STATE $ \widehat \rho^{(i)}_j \leftarrow  \rho^{(i)}_j $
		\STATE $ \widehat \xi^{(i)}_j \leftarrow  \xi^{(i)}_j $
		\STATE $\mathtt{flag_{reception}}\leftarrow 0$
		\STATE $\mathtt{flag_{update}}\leftarrow 1$ (optional)
		\ENDIF	
		\Statex{ \bf Estimate update} (atomic)
		\IF {$\mathtt{flag_{update}}=1$} 
		\STATE $\widehat  \rho^{(i)}_i \leftarrow \nabla_i J_i(x_i,\{\widehat x^{(i)}_k\}_{k\in\N_i})$
		\STATE $\widehat \xi^{(i)}_i \leftarrow \nabla^2_{ii} J_i(x_i,\{\widehat x^{(i)}_k\}_{k\in\N_i}) $
		\STATE $x_i \leftarrow x_i - \epsilon \big(\sum_{j\in\mathcal{N}_i^+}  \widehat \xi^{(i)}_j \big)^{-1}\big(\sum_{j\in\mathcal{N}_i^+} \widehat \rho^{(i)}_j \big)$
		\STATE $\mathtt{flag_{update}}\leftarrow 0$
		\STATE $\mathtt{flag_{transmission}}\leftarrow 1$ (optional)
		\ENDIF		
	
\end{algorithmic}
\caption{Resilient Block Jacobi (RBJ) Algorithm (node $i$)}
\label{alg:rbj}
\end{algorithm}
\begin{remark}[{\bf Resilient gradient descent (RGD) Algorithm}]\label{rem:rgd}
If memory, communication and computational complexity are a concern, it is possible to modify the proposed algorithm mimicking the standard gradient descent algorithm. In this framework, the second order information is not needed and therefore the variables $\xi_j^{(i)},\widehat \xi_j^{(i)}$ (\texttt{lines 5, 6, 17, 21} in Algorithm~\ref{alg:rbj}) do not need to be computed and the update for the local variable $x_i$ (\texttt{line 22} in Algorithm~\ref{alg:rbj}) should be replaced with the following:
$$
x_i \leftarrow x_i - \epsilon \sum_{j\in\mathcal{N}_i^+} \widehat \rho^{(i)}_j\, . 
$$
Obviously, the price to pay for this choice is a likely decrease in convergence speed. 
\end{remark}

\begin{remark}[{\bf Resilient Weighted Least Squares (RWLS) Algorithm}]\label{rem:rwls}
If the local cost functions are quadratic, i.e:
$$
J_i(x_i,\{x_j\}_{j\in\mathcal{N}_i})=\frac{1}{2} \|y_i-A_i x\|_{W_i}^2= \frac{1}{2}(y_i-\sum_{j\in\mathcal{N}_i^+}A_{ij}x_j)\T W_i(y_i-\sum_{j\in\mathcal{N}_i^+}A_{ij}x_j)\, , 
$$
where $W_i>0$ are the local weights, then the problem to be solved becomes a  Weighted Least Squares problem. For this special case, the gradient and the hessian components simplify to:
\begin{equation}\label{eq:local_memory_wls}
\rho^{(j)}_i(x) :=A_{ij}\T W_i(\sum_{j\in\mathcal{N}_i^+}A_{ij}x_j - y_i)\, ,\qquad\qquad
\xi^{(j)}_i(x) := A_{ij}\T W_iA_{ij}\, ,
\end{equation}
therefore the RBJ Algorithm can be simplified by substituting \texttt{lines 10} and \texttt{11} with the following updates:
\begin{eqnarray}
 \rho^{(j)}_i &\leftarrow& A_{ij}\T W_i(A_{ii}x_i + \sum_{j\in\mathcal{N}_i}A_{ij}\widehat x^{(i)}_j - y_i), \ \ \forall j\in \mathcal{N}_i \, ,\\
\xi^{(j)}_i &\leftarrow&  A_{ij}\T W_iA_{ij}\, . 
\end{eqnarray}
It is clear from the previous expression, that the algorithm could be modified by having a preliminary phase when the $\xi^{(j)}_i$ are transmitted reliably to the neighbours so that eventually $\widehat \xi^{(j)}_i =\xi^{(j)}_i$, and then the algorithm could simply transmit the variables $x_i,\rho^{(j)}_i$ and update the variables $x_i,\widehat x^{(i)}_j, \widehat \rho^{(i)}_j$  which are the only variables  that evolve over time, thus considerably reducing the communication complexity which corresponds with that of the RGD algorithm. 
\end{remark}

\subsection{Theoretical analysis of RBJ Algorithm}

\noindent Before presenting the major theoretical result characterizing the convergence properties of the proposed RBJ algorithm, we introduce the following assumption on the nature of lossy communication we consider. It mainly states that each agent $i\in\V$ receives information coming from each agent $j\in\N_i$ at least once within any window of $T$ iterations of the algorithm.
\begin{ass}[{\bf Persistent communication}]\label{ass:persistent_comm}
There exists a constant $T$ such that, for all $t\geq 0$, for all $ i \in \mathcal V$ and for all $j \in \mathcal N_i$, 
$$
\mathbb P \left[\{\gamma^{(i)}_j(t),\dots,\gamma^{(i)}_j(t+T)\} = \{0, \dots, 0\}\right] = 0 .
$$
\hfill$\square$
\end{ass}
\begin{theorem}[{\bf Local convergence of the RBJ algorithm}]\label{thm:conv_rbj}
Let Assumptions \ref{ass:convex} and \ref{ass:persistent_comm} hold. Moreover assume that the cost functions $J_i$ are three-time differentiable and continuous. Consider Problem \eqref{eq:ConvexProblem} and the RBJ algorithm. Let $x^*$ be the minimizer of
\eqref{eq:ConvexProblem}. There exists $\bar \epsilon>0$ and $\delta>0$, such that, if $0<\epsilon < \bar \epsilon$ and $\|x(0)-x^*\| < \delta$, 
then the trajectory $x(t)$, generated by the RBJ algorithm, converges exponentially fast to $x^*$, i.e.,  
$$
\| x(t)-x^*\| \leq C \rho^t
$$
for some constants $C>0$ and $0 < \rho < 1$.\hfill$\square$
\end{theorem}
\noindent The proof of Theorem \ref{thm:conv_rbj} can be found in Appendix \ref{apx_proofThm}, and basically relies on separation of time scales principle between the dynamics of the states $x_i$'s and those of the auxiliary variables $\widehat{x}_j^{(i)}$'s, $\rho^{(i)}_j$'s, $\widehat{\rho}_i^{(j)}$'s, $\xi^{(i)}_j$'s and $\widehat{\xi}_i^{(j)}$'s.
Loosely speaking, the result builds on the idea that if the step-size $\epsilon$ is small enough, the variation of the true states $x_i$'s is sufficiently slow and, despite the lossy communication, the values of the auxiliary variables stored in memory equal the true values.

\begin{remark}[{\bf Local convergence of the RGD algorithm}]
The same argument used in the previous theorem can be applied to the Robust Gradient Descent Algorithm presented in Remark~\ref{rem:rgd} above under the weaker assumption that the cost functions $J_i$ are two-time differentiable, thus providing the same local exponential convergence. Typically, the critical value $\bar \epsilon$ for the RGD algorithm is smaller than that of the RBJ algorithm, and consequently also the rate of convergence is slower.\hfill$\square$ 
\end{remark}

\begin{lemma}[{\bf Global convergence RWLS algorithm}]\label{prop:conv_rj_quadratic}
Let Assumptions \ref{ass:convex} and \ref{ass:persistent_comm} hold. Consider Problem \eqref{eq:ConvexProblem} with a quadratic cost function $J(x)$ and the RWLS algorithm. There exists $\bar \epsilon$ such that, if $0<\epsilon < \bar \epsilon$, then, for any $x(0)\in \R^{n}$, the trajectory $x(t)$, generated by the RWLS algorithm, converges exponentially fast to the minimizer $x^*$ of the corresponding problem, i.e.,  
$$
\| x(t)-x^*\| \leq C \rho^t
$$
for some constants $C>0$ and $0 < \rho < 1$.\hfill$\square$
\end{lemma}
%

\section{Simulations}\label{sec:simulations}
In this section we present some simulative results obtained using the RBJ algorithm. The simulations involve the IEEE 123 nodes distribution grid benchmark (see \cite{BS:CR:TM:2014}). The problem we address is the robust estimation of the voltage level at each node of the grid (except the PCC node which is assumed fixed and known) from voltage and current measurements in the presence of measurements outliers. We recall that voltages and currents in an AC power distribution grid are complex values. However, in view of the state estimation problem we consider, it is convenient to exploit an equivalent standard reformulation in rectangular coordinates. In particular, given the complex vectors of voltages and currents, denoted as $\mathbf{v}\in\C^{122}$ and $\mathbf{i}^c\in\C^{122}$ respectively, and the weighted Laplacian matrix $\mathcal{L}\in\C^{122\times 122}$ describing the electric grid, thanks to Kirchhoff's voltage and current laws, it holds that
\begin{equation}\label{eq:kirchhoff}
\mathbf{i}^c= \mathcal{L} \mathbf{v}.
\end{equation}
However, by rewriting voltage and currents in rectangular coordinates as 
$$
v:=[\Re(\mathbf{v})\T \ \Im(\mathbf{v})\T ]\T \in\R^{244}\, ,\quad 
i^c := [\Re(\mathbf{i}^c)\T \ \Im(\mathbf{i}^c)\T ]\T \in\R^{244}\, .
$$
and, similarly, by splitting $\mathcal{L}$ into its real and imaginary parts as 
\[
L=\begin{bmatrix}
\Re(\mathcal{L}) & -\Im(\mathcal{L}) \\ \Im(\mathcal{L}) & \Re(\mathcal{L})
\end{bmatrix}\, ,
\]
Eq.~\eqref{eq:kirchhoff} is equivalent to 
$$
i^c = Lv\, .
$$
Thus, by assuming to collect both current and voltage measurements directly in rectangular coordinates\footnote{According to future smart grids paradigm, it is assumed each node of the grid to be equipped with a smart measurement units, e.g., a Phasor Measurement Unit (PMU), which can return measurements of current and voltage. Usually, electric quantities are measured in polar coordinates. However, for the sake of simplicity, we assume to have at our disposal measurements directly in rectangular coordinates, stressing that, thanks to a suitable linearization, it is always possible to pass from polar to rectangular coordinates.}, our measurement model reads as
\[\begin{bmatrix}
y^v \\y^{ i^c}
\end{bmatrix} = \begin{bmatrix}
I \\ L
\end{bmatrix}v + 
\begin{bmatrix}
w^v \\w^{i^c}
\end{bmatrix} + 
\begin{bmatrix}
o^v \\ o^{i^c}
\end{bmatrix}\, ,\qquad
\begin{bmatrix}
w^v \\w^{i^c}
\end{bmatrix} \sim 
\N\left(
\begin{bmatrix}
0 \\ 0
\end{bmatrix},
\begin{bmatrix}
\sigma_v^2\diag(|v|)  & \\ & \sigma^2_{i^c}\diag(|i^c|)
\end{bmatrix}
\right)\, ,
\]
where $I\in\R^{244}$ is the identity matrix, $y^v,y^{i^c}\in R^{244}$ are the measurements, which we collect in vector $y\in\R^{488}$, $w^v,w^{i_c}\in\R^{244}$ are the measurements' noise, and  $o^v,o^{i^c}\in\R^{244}$ are sparse vectors which contain possible measurement outliers. We choose\footnote{The choice for the measurements error standard deviations is dictated by the fact that the de facto standard for modern PMUs requires at most a $0.1\%$ error in the voltage measurements. This translates in a current error of more or less $10\%$.} $\sigma_v=10^{-3}$[p.u.] and $\sigma_{i^c}=10^{-1}$[p.u.].  
Finally, concerning the outliers, $10\%$ of the measurements are corrupted, and the distribution of the outliers is uniform between $1/100$ and $1/80$ of the respective measurement for voltages and between $1/2$ and $1$ of the respective current measurement.\\ 
As suggested at the end of Section~\ref{sec:example}, to perform robust state estimation in the presence either of measurements faults or outliers, one interesting choice for the cost function is the modified 1-norm defined is Eq.~\eqref{eq:modified1norm} as
$$
\robustCost{r}
$$
where $r=y-Av$ are the measurements residuals with $A = [I\,\,\,\, L\T ]\T $.
To run the RBJ algorithm we need to identify some partition of the grid. To do so, the feeder is divided into $N$ non overlapping areas, and a computing unit, which can collect the measurements of the nodes belonging to the area and can run the algorithm, is associated to each area. An example of the division in areas is given in Figure \ref{fig:areasDivision}. The communication graph $\G$ can be obtained from the division in areas, and in particular, two units can communicate with each other if the  two areas are physically connected (that is if there exists two nodes, each one belonging to one of the areas, which are connected by an electric wire). The vectors $y,x$ and $r$ and matrix $A$ are divided according to the partition of the nodes in areas. 
Given the cost function, the values of $\rho_j^{(i)}$ and $\xi_j^{(i)}$ are computed as:
\begin{align*} \rho_j^{(i)} &= {A_{ji}^\top \left(\left(\diag(r_j(t))\right)^2\! +\! \nu I\right)^{-1/2}r_{j}(t)}, \quad \forall\ j\in\N_i^+\\
\xi_j^{(i)} & = \nu H_{ji}\T  \left(\left(\diag(r_j(t))\right)^2 + \nu I \right)^{-3/2} H_{ji}, \quad \forall\ j\in\N_i^+.
\end{align*}
We tested the RBJ algorithm under two different scenarios to evaluate the influence of different parameters involved in the algorithm. All the results showed are obtained averaging over 100 Monte Carlo runs (MCR). 
\begin{figure}
\begin{center}
\includegraphics[scale=0.7]{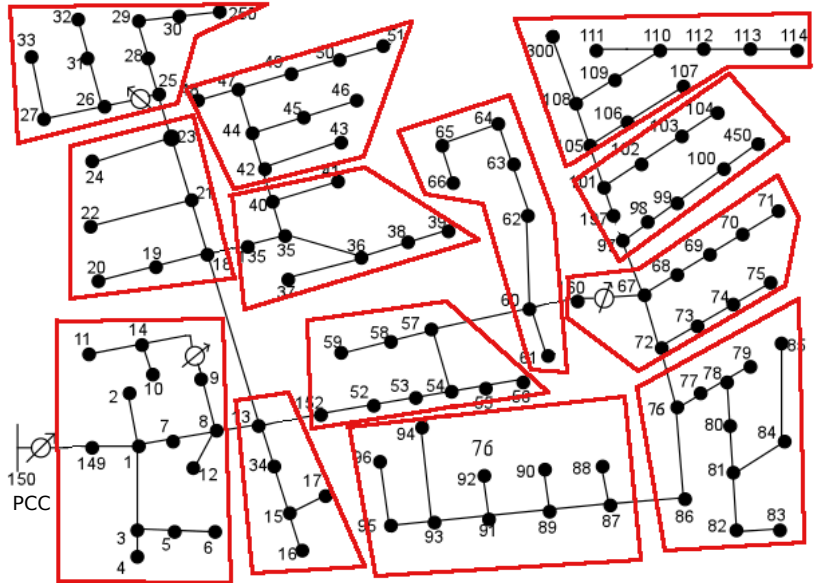}
\caption{Division in $N = 13$ areas of the IEEE 123  nodes distribution grid.}
\label{fig:areasDivision}
\end{center}
\end{figure}

In the first considered scenario we study the influence of the number $N$ of areas on the performance of the algorithm. Observe that, for the case $N=1$ the proposed RBJ algorithm resembles a Newton-Raphson iteration. Thus, in general and as shown in Figure~\ref{fig:differentN}, the fewer the number of areas, the faster the convergence rate.
\begin{figure}
\begin{center}
\includegraphics[scale=0.8]{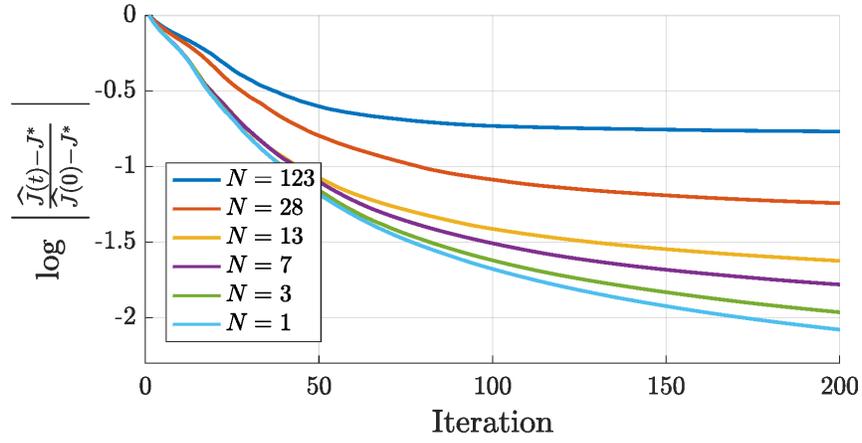}
\caption{Normalized cost function as a function of the iteration, for different numbers of areas $N$. In all simulations there is a packet loss probability of $30\%$ and $\epsilon=0.0004$. The results are obtained averaging over 100 MCR.}
\label{fig:differentN}
\end{center}
\end{figure}
In the second scenario we analyze the influence of the step size $\epsilon$ on the convergence rate. As can be seen from Figure \ref{fig:differentEps}, we can infer that the convergence rate improves for bigger values of $\epsilon$. Nevertheless, it is important to highlight the fact that that the algorithm may diverge if the selected $\epsilon$ is a too large.
\begin{figure}
\begin{center}
\includegraphics[scale=.8]{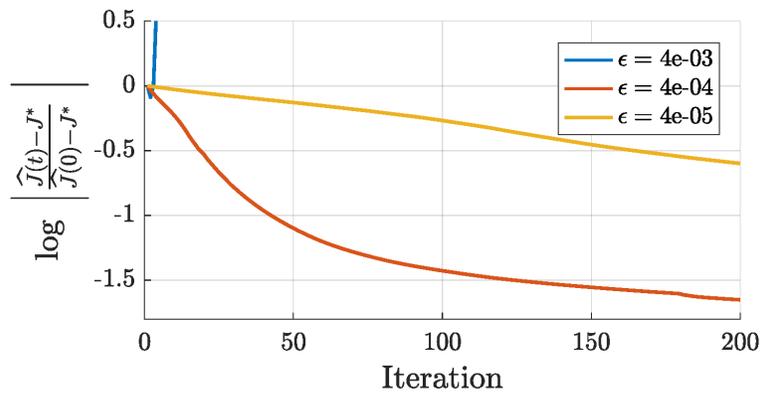}
\caption{Normalized cost function as a function of the iteration, for different values of the parameter $\epsilon$. The number of areas used is $N = 13$ and the results are obtained averaging over 100 MCR. In all simulations there is a packet loss probability of $30\%$.}
\label{fig:differentEps}
\end{center}
\end{figure}
Finally, we have considered a third scenario, not reported reported here, in  which the robustness of the algorithm is tested for increasing values of the packet loss probability. The results show that the algorithm is really robust to packet losses as can be seen both from Figures~\ref{fig:differentN}--\ref{fig:differentEps} where the algorithm has been tested considering a $30\%$ packet loss probability. In particular, since the curves obtained are really close to each other, we decided not to show the results of the simulations. As last remark regarding the packet loss scenario, it is numerically observed that the higher the packet loss probability, the smaller the value of the step size to ensure convergence of the algorithm. Therefore, in the choice of the step size of the algorithm, which is still an open problem, the degree of reliability of the communication network must be considered.

\section{Conclusions}\label{sec:conclusion}
Considering the emerging area of large-scale multi-agent systems, this paper addressed the always more timely problem of unconstrained robust distributed convex optimization in the presence of communication non idealities. In particular, we analyzed a particular class of locally coupled cost functions which arise in diverse interesting engineering problems such as multi-area state estimation of smart electric grids and multi-robot localization, just to mention two possible applications. We considered a particularly flexible partition-based communication architecture which seamlessly accounts for peer-to-peer and wide-area communications. We proposed a generalized gradient algorithm based on the well-known Jacobi iteration. By leveraging Lyapunov theory and separation of time scale principle, we proved robustness of the algorithm to packet drops and communication failures. Finally, we extensively tested the proposed solution for robust state estimation in the presence of measurements outliers using as benchmark the standard IEEE 123 nodes distribution feeder.

\appendix
\section{Proof of Theorem \ref{thm:conv_rbj} }\label{apx_proofThm}
The proof of Theorem \ref{thm:conv_rbj} relies on the time scale separation of the dynamic of the $x_i$'s and of the auxiliary variables $\widehat x_j^{(i)}$'s, $\widehat \rho_j^{(i)}$'s and $\widehat \xi_j^{(i)}$'s, and fully exploits the following Lemma
\begin{lemma}[Time scale separation principle for discrete time dynamical systems]\label{lemma}
Consider the dynamical system
\begin{equation}
\begin{bmatrix}
x(t+1)\\y(t+1)
\end{bmatrix} = 
\begin{bmatrix}
I& -\epsilon B\\ C(t) & F(t)
\end{bmatrix} 
\begin{bmatrix}
x(t)\\y(t)
\end{bmatrix}.
\label{eq:dynsys_lem}
\end{equation}
Let the following assumptions hold
\begin{enumerate}
\item There exists a matrix $G$ such that $y = Gx$ satisfies the expression $y = C(t) x + F(t) y,\  \forall t,\forall x$
\item the system 
\begin{equation}
z (t+1) = F(t) z(t) 
\label{eq:z(t+1)}
\end{equation}
is exponentially stable;
\item the system 
\begin{equation}
\dot x(t) = - B G x(t)
\label{eq:xi}
\end{equation}
is exponentially stable.

\item The matrices $C(t)$ and $F(t)$ are bounded, i.e. there exists $m>0$ such that $\|C(t)\|<m, \|F(t)\|<m, \forall t\geq 0$.

\end{enumerate}
Then, there exists $\bar \epsilon$, with $0<\epsilon<\bar \epsilon$ such that the origin is an exponentially stable equilibrium for the system \eqref{eq:dynsys_lem}. \hfill$\square$
\end{lemma}
\begin{proof}[Proof of Lemma~\ref{lemma}]
Let us first consider the following change of variable:
$$ z(t)=y(t)-Gx(t)$$ 
The dynamics of the system in the variables $x,z$ can be written after some straightforward manipulations as follows:
\begin{equation}
\begin{bmatrix}
x(t+1)\\z(t+1)
\end{bmatrix} = 
\left( \underbrace{\begin{bmatrix}
I-\epsilon BG &0\\ 0 & F(t)
\end{bmatrix}}_{\Sigma(t)} + \epsilon \underbrace{\begin{bmatrix}
0 &-BG\\ GBG & GB
\end{bmatrix}}_{\Gamma}  \right)
\underbrace{\begin{bmatrix}
x(t)\\z(t)
\end{bmatrix}}_{\mu(t)}
\end{equation}
where we used Assumption~1. From Assumption~2, 3 and 3, using converse Lyapunov theorems \cite{Khalil:96}, it follows that there exist positive definite matrices $P_x>0$ and $P_z(t)>0$ such that 
$$ -P_xBG-G^TB^TP_x\leq -a I, \ \ F(t)^TP_z(t+1)F(t)-P_z(t)\leq -a I, \forall t $$ 
where $a$ is a positive scalar and $P_z(t)$ is bounded, i.e. $\|P_z(t)\|\leq m$.
We will use the following positive definite Lyapunov function to prove exponential stability of the whole system:
$$ U(x,z,t) = x^TP_xx+z^TP_z(t)z = \begin{bmatrix}
x^T & z^T 
\end{bmatrix}  \underbrace{\begin{bmatrix}
P_x &0\\ 0 & P_z(t)
\end{bmatrix}}_{P(t)} \begin{bmatrix} x \\ z\end{bmatrix}$$
If we define time difference of the Lyapunov function as $\Delta U(x,z,t)= U(x(t+1),z(t+1),t+1)-U(x(t,)z(t),t)$ we get:
\begin{eqnarray*} 
\Delta U(x,z,t)\!\!\!&=&\!\! x^T\left(-\epsilon(P_xBG\!+\!G^TB^TP_x)\! +\!\epsilon^2 G^TB^TP_xBG\right)x\!+\\  
\!\!\!&&\!\! +z^T\!\!\left(F(t)^TP_z(t\!+\!1)F(t)\!-\!P_z(t)\right)z \!+\!2\epsilon\mu^T\Sigma^T(t)P(t\!+\!1)\Gamma\mu \!+\!\epsilon^2 \mu^T \Gamma^TP(t\!+\!1)\Gamma \mu \\
\!\!\! &\leq &\!\! \!\!-\epsilon a \|x\|^2\! \!-\! \!a \|z\|^2\! \!+\!\epsilon^2 \!\underbrace{\|P^{\frac{1}{2}}_xBG\|^2}_{b}\! \|x\|^2 \!\!+\!2\epsilon\mu^T\Sigma^T\!(t)P(t\!+\!1)\Gamma\mu \!+\!\epsilon^2 \|P^{\frac{1}{2}}(t\!+\!1)\Gamma\|^2\|\mu\|^2
\end{eqnarray*}
Note that the top left block of $\Gamma$ is zero and that $\Sigma(t)$ and $P(t)$ are diagonal and bounded for all times. From this it follows that
$$ \Sigma^T(t)P(t+1)\Gamma = \begin{bmatrix}
0 &\star \\ \star & \star
\end{bmatrix}    \Longrightarrow  2 \mu^T \Sigma^T(t)P(t+1)\Gamma\mu \leq c (2\|x\|\|z\|+\|z\|^2) $$
for some positive scalar $c$. Boundedness of $P(t)$ also implies that 
$$ \|P^{\frac{1}{2}}(t+1)\Gamma\|^2\|\mu\|^2 \leq d (\|x\|^2+\|z\|^2) $$
for some positive scalar $d$. Putting all together we get
$$ \Delta U(x,z,t) \leq  \begin{bmatrix}
\begin{bmatrix}\|x\| & \|z\| \end{bmatrix}
\end{bmatrix}  \begin{bmatrix}
-\epsilon a +b\epsilon^2 & \epsilon c \\ \epsilon c & -a +\epsilon c +\epsilon^2 d
\end{bmatrix} \begin{bmatrix} \|x\| \\ \|z\|\end{bmatrix} $$
It follows immediately that there exists a critical $\overline \epsilon$ such that for $0<\epsilon<\overline \epsilon$ the matrix in the above equation is strictly negative definite and therefore the system is exponentially stable. 
\end{proof}

We are now ready to state the formal proof of Theorem~\ref{thm:conv_rbj}.
\begin{proof}[Proof of Theorem \ref{thm:conv_rbj}]
The proof relies on Lemma~\ref{lemma}. In order to improve readability, this proof is broken into few steps. The first step is to write the evolution of the RBJ algorithm as the evolution of a dynamical system. The second step is to find its equilibrium point and to linearize it around this point. The third step is show that the linerized dynamical system satisfies the three assumptions listed in Lemma~\ref{lemma}.\\

\noindent \emph{RBJ as a dynamical system:} \\
First of all, note that thanks to Assumption \ref{ass:convex} the second order derivatives and in particular all the variables $\xi^{(i)}_j,\widehat{\xi}^{(i)}_j$ are always well defined and invertible. Now, let the vectors  $\widehat e^{(i)}_j$ be the \emph{vectorization} of  $\widehat \xi^{(i)}_j$, $\widehat e^{(i)}_j = \mathrm{vec}(\widehat \xi^{(i)}_j)$, and the \emph{un-vectorization operator} $\mathrm{vec}^{-1}$ as the inverse of the vectorization operator, i.e.  $ \mathrm{vec}^{-1}(\widehat e^{(i)}_j) = \widehat \xi^{(i)}_j $. 
Let $\widehat x_i$, $\widehat \rho_i$ and $\widehat e_i$ be the vectors in which all the $\widehat x^{(i)}_j$'s, the $\widehat \rho^{(i)}_j$'s, and the $\widehat e^{(i)}_j$'s are stacked, respectively, i.e. $\widehat x_i= (\widehat x^{(i)}_{j_1} \cdots \widehat x^{(i)}_{j_{\mathcal N_i}}) $ and similarly for $\widehat \rho_i$ and $\widehat e_i$.
Let $x$, $\widehat x$, $\widehat \rho$, $\widehat e$ be the vectors collecting all the $x_i$, $\widehat x_i$'s, $\widehat \rho_i$'s and $\widehat e_i$'s, respectively, i.e. $x = (x_1 \cdots  x_N)$ and similarly for $\widehat x$, $\widehat \rho$ and $\widehat e$.\\
For every agent $i$ and neighbours $j\in\mathcal{N}_i$, the dynamic of the local variables are given by the following equations:
\begin{subequations}
\begin{align}
x_i (t+1) & = f_1^i(x(t), \widehat \rho(t), \widehat e(t))\\
\widehat x^{(i)}_j(t+1) & = f_2^{ij} (x(t),\widehat x(t), t)\\
\widehat \rho^{(i)}_j(t+1) & = f_3^{ij} (x(t),\widehat x(t),\widehat \rho(t),  t)\\
\widehat e^{(i)}_j(t+1) & = f_4^{ij} (x(t),\widehat x(t),\widehat e(t),  t)
\end{align}
\label{eq:system_update}
\end{subequations}
where
\begin{subequations}
\begin{align}
f_1^{i}(x,\widehat \rho, \widehat e) &= x_i - \epsilon \Big( \underbrace{\sum_{j \in \mathcal N_i^+} \mathrm{vec}^{-1} (\widehat e^{(i)}_j)}_{f^i_{e}(\widehat e)} \Big) ^{-1} \Big( \underbrace{\sum_{j \in \mathcal N_i^+} \widehat \rho^{(i)}_j}_{f^i_\rho(\widehat \rho)} \Big) \\
f_2^{ij}(x,\widehat x, t) & = 
\begin{cases}
\widehat x^{(i)}_j \;  \quad  \quad  \quad\text{ if } \gamma^{(i)}_j(t) = 0 \\
x_j \quad  \quad  \quad  \quad \text{ if } \gamma^{(i)}_j(t) = 1
\end{cases} \\
f_3^{ij}(x,\widehat x, \widehat \rho, t) & = 
\begin{cases}
\widehat \rho^{(i)}_j \;  \quad  \quad  \quad \quad  \quad  \quad  \quad \text{ if } \gamma^{(i)}_j(t) = 0 \\
 \nabla_i J_j(x_j,\{\widehat x^{(j)}_k\}_{k\in\N_j})  \quad  \text{ if } \gamma^{(i)}_j(t) = 1 
\end{cases}\\
f_4^{ij}(x,\widehat x, \widehat e, t) & = 
\begin{cases}
\widehat e^{(i)}_j \qquad  \text{ if } \gamma^{(i)}_j(t) = 0 \\
\mathrm{vec} \left( \nabla^2_{ii} J_j(x_j,\{\widehat x^{(j)}_k\}_{k\in\N_j}) \right) \; \text{ if } \gamma^{(i)}_j(t) = 1 
\end{cases}.
\end{align}
\end{subequations}
Note that the variables $\rho^{(i)}_j$ and $\xi^{(i)}_j$ do not appear in the dynamics since they are deterministic functions of the variables $x$ and $\widehat x$, and therefore can be omitted. \\

\noindent \emph{Equilibrium point and linearization:} \\
Let $x^*$ be the minimizer of the optimization problem and let us define
\begin{align*}
H_{hk} &= \nabla^2_{hk} J(x^*) = \sum_{j=1}^N  \underbrace{\nabla^2_{hk} J_j(x^*_j,\{ x^*_k\}_{k\in\N_j})}_{H_{hk}^j}= \sum_{j=1}^N H_{hk}^j\\
\widehat x^{(i)*}_j &= x^*_j \\
\widehat \rho^{(i)*}_j &= \nabla_i J_j(x^*_j,\{ x^*_k\}_{k\in\N_j})\\
\widehat e^{(i)*}_j &= \mathrm{vec} \left( \nabla^2_{ii} J_j(x^*_j,\{ x^*_k\}_{k\in\N_j})\right)= \mathrm{vec}(H^j_{ii})
\end{align*}
Notice that $\sum_{j=1}^N \widehat \rho^{(i)*}_j = \nabla_i J(x^*)=0$, since the gradient computed at the minimizer is zero. It is now simple to verify by direct inspection that $(x^*, \widehat x^*, \widehat \rho^*, \widehat e^*)$ is an equilibrium point for the dynamical system described by \eqref{eq:system_update}.
Next, we will analyze the behaviour of system \eqref{eq:system_update} in the neighborhood of the equilibrium point $(x^*, \widehat x^*, \widehat \rho^*, \widehat e^*)$.
Consider the change of variables
\begin{equation}
\begin{array}{rl}
\psi &= x - x^* \\
\widehat \psi &= \widehat x - \widehat x^*\\
\widehat \eta &= \widehat \rho - \widehat \rho^{*} \\
\widehat \zeta & = \widehat e - \widehat e^*
\end{array}
\label{equ:chng_var}
\end{equation}
If we linearize equations \eqref{eq:system_update} around $(x^*, \widehat x^*, \widehat \rho^*, \widehat e^*)$, we obtain
\begin{align}
\psi_i(t+1) &\simeq \psi_i(t) - \epsilon H_{ii}^{-1} \, \sum_{j\in\mathcal{N}_i^+} \widehat \eta^{(i)}_j \label{1}\\
\widehat \psi_j^{(i)}(t+1) & \simeq 
\begin{cases}
\widehat \psi_j^{(i)}(t) \; \text{ if } \gamma^{(i)}_j(t) = 0 \\
\psi_j(t) \;\text{ if } \gamma^{(i)}_j(t) = 1
\end{cases} \\
\widehat \eta^{(i)}_j (t+1) & \simeq 
\begin{cases}
\widehat \eta^{(i)}_j (t) \; \text{ if } \gamma^{(i)}_j(t) = 0 \\
H^j_{ij} \psi_j (t) +\sum_{k\in\mathcal{N}_j}H^j_{ik} \widehat \psi_k^{(i)}(t)    \text{ if } \gamma^{(i)}_j(t) = 1 \label{3}
\end{cases}\\
\widehat \zeta^{(i)}_j (t+1) & \simeq 
\begin{cases}
\widehat \zeta^{(i)}_j (t) \; \text{ if } \gamma^{(i)}_j (t) = 0 \\
K^j_{ij} \psi_j (t) +\sum_{k\in\mathcal{N}_j}K^j_{ik} \widehat \psi_k^{(i)}(t) \; \text{ if } \gamma^{(i)}_j (t) = 1 . \label{2}
\end{cases}.
\end{align}
where in Eqn.~\eqref{1} we used the fact that $\left.\frac{\partial f_1^i}{\partial \widehat e}\right|_{x^*,\widehat \rho^*,\widehat e^*}=\left.-\epsilon\frac{\partial (f_e^i)^{-1}}{\partial \widehat e}f^i_\rho\right|_{x^*,\widehat \rho^*,\widehat e^*}=0$  since $\left.f^i_\rho\right|_{x^*,\widehat \rho^*,\widehat e^*}=\nabla_i J(x^*)= 0$, and the fact that $f_e^i(\widehat e^*)=H_{ii}$. In Eqn.\eqref{3} we used the fact that $H^j_{ik}=\nabla^2_{ik}J_j(x^*_j,\{ x^*_k\}_{k\in\N_j})$. Finally, in Eqn.~\eqref{2} the matrices $K^j_{ik}$ depends on third order derivatives of $J(x)$ whose values are unimportant for the analysis of the stability of the dynamics.
By collecting all the variables together, we obtain the system 
%
\begin{align}
\begin{bmatrix}
\psi(t+1)\\ \widehat \psi(t+1) \\ \widehat \eta(t+1) \\ \widehat \zeta (t+1)
\end{bmatrix} &=
\left[
\begin{array}{c|ccc}
I& 0& -\epsilon B& 0 \\ \hline
C_1(t) & F_1(t) & 0 & 0\\
C_2(t) & F_2(t) & F_3(t) & 0\\
C_3(t) & F_4(t) & 0 & F_5(t) \\
\end{array}
\right]
\begin{bmatrix}
\psi(t)\\ \widehat \psi(t) \\ \widehat \eta(t) \\ \widehat \zeta (t)
\end{bmatrix}			\notag \\
\begin{bmatrix}
\psi(t+1)\\ y(t+1)
\end{bmatrix}&=
\left[
\begin{array}{c|c}
I& -\epsilon B \\ \hline
C(t) & F(t) 
\end{array}
\right]
\begin{bmatrix}
\psi(t)\\ y(t)
\end{bmatrix}.
\label{eq:dynamic_optalg}
\end{align}
\noindent where $y=(\widehat \psi, \widehat \xi, \widehat \zeta)$ collects the fast dynamic variables. Notice that $F_1(t)$, $F_3(t)$ and $F_5(t)$ are diagonal matrices whose entries are either 1 or 0, depending on the communication between agent success, and, as a consequence, $F(t)$ is a lower triangular matrix, $\forall t$.\\

\noindent \emph{Assumption 1 of Lemma~\ref{lemma}:} \\
We now start proving that the linearized dynamics above satisfies the three assumptions of Lemma~\ref{lemma} where $\psi$ plays the role of $x$ in the Lemma. 
It is simple to verify by direct inspection that for a fixed $\psi$, the following maps satisfy Assumption 1 of Lemma~\ref{lemma}:
\begin{align}
\widehat \psi^{(i)}_j  &=  \psi_j \\
\widehat \eta^{(i)}_j & = H^j_{ij} \psi_j  +\sum_{k\in\mathcal{N}_j}H^j_{ik} \psi_k  \\
\widehat \zeta^{(i)}_j  & = K^j_{ij} \psi_j  +\sum_{k\in\mathcal{N}_j}K^j_{ik} \psi_k
\end{align}
in fact, this is equivalent of saying that there exists a matrix $G$ such that $y=G\psi$ satisfies the equality $y=C(t)\psi + F(t)y$ for all $\psi$ and $t$.\\

\noindent \emph{Assumption 2 of Lemma~\ref{lemma}:} \\%
Let us now consider the fast dynamics of the system given by the following system: 
\begin{equation*}
z(t) = F(t-1)\cdots F(0) z(0) = \Omega(t)z(0) 
\end{equation*}
Assumption~\ref{ass:persistent_comm}, on the persistent communication among the agents, assures that
$$
\begin{array}{l}
F_1(T-1)\cdots F_1(0) = \Omega_1(T)=0 \\
F_3(T-1)\cdots F_3(0) =  \Omega_3(T)=0 \\
F_5(T-1)\cdots F_5(0) =  \Omega_5(T)=0 
\end{array}
$$
in fact when $\gamma_{j}^{(i)}(t)=1$, the corresponding raws in the matrices $F_1(t),F_3(t),F_5(t)$ become zero, and this property will be inherited also by the product matrices  $\Omega_1(T)$, $\Omega_3(T)$, $\Omega_3(T)$ since all $F_1(t)$, $F_2(t)$, $F_3(t)$ are diagonal. Since all $\gamma_{j}^{(i)}(t)$ will be equal to one at least once within the window $t\in[0,\cdots, T-1]$, then the matrices  $\Omega_1(T)$, $\Omega_3(T)$, $\Omega_3(T)$ must be all zero.  
Finally, since the matrix $F(t)$ is lower triangular, we have that after a maximum of $(2T + 1)$ iterations the product matrix  $\Omega(2T+1)$ will be zero and thus $z(2T + 1) = 0$. That is, the fast variable dynamic is exponentially stable, since it reaches the equilibrium in a finite number of iteration.\\

\noindent \emph{Assumption 3 of Lemma~\ref{lemma}:} \\%
Finally, consider the slow dynamical system 
\begin{equation}
\dot \psi(t) = - B G \psi (t).
\label{eq:xi2}
\end{equation} 
which by direct substitution from the previous analysis can be locally written as:
$$ \dot \psi_i(t) = -H^{-1}_{ii} \left(\sum_{j\in\mathcal{N}_i^+}\left(H^j_{ij} \psi_j  +\sum_{k\in\mathcal{N}_i}H^j_{ik} \psi_k \right)\right)=-H^{-1}_{ii}H^i\psi $$
where $H$ was defined above and corresponds to the Hessian of the global cost $J$ computed at $x^*$, i.e. $H=\nabla^2 J(x^*)$ and $H^i$ is its $i$-th block-row,\\ i.e., $H^i= [\nabla^2_{i1} J(x^*)\cdots \nabla^2_{iN} J(x^*)]$. This implies that
$$BG = \left(\mathrm{diag}(H)\right)^{-1}H\, ,$$
therefore, if we choose
\begin{equation*}
V(\psi) = \frac{1}{2} \psi\T  H\psi,
\end{equation*}
as a Lyapunov function, it is straightforward to see that system \eqref{eq:xi2} is asymptotically stable since $\dot V(\psi(t)) = -\psi\T (t) H\left(\mathrm{diag}(H)\right)^{-1}H \psi(t)<0,x\neq 0$ being $H>0$ by assumption.

\noindent \emph{Assumption 4 of Lemma~\ref{lemma}:} \\
This comes from the observation that the time-variance of the state matrices depends on the specific sequence of packet losses that can occur. Since there are only a finite number of possible different sequences, the assumption is clearly satisfied. 

Concluding, system \eqref{eq:dynamic_optalg} satisfies the hypothesis of Lemma \ref{lemma}, and thus there exists $\bar \epsilon$, with $0<\epsilon<\bar \epsilon$ such that, by using the resilient block Jacobi Algorithm \ref{alg:rbj},
$$\lim_{t\rightarrow \infty} x(t) = x^*.$$
locally exponentially fast.

\end{proof}

\bibliographystyle{IEEEtran}
\bibliography{Biblio.bib}

\end{document}